\documentclass[]{article}
\usepackage{amsmath,amssymb,amsthm,mathrsfs,graphicx,subfigure,float,
caption,epstopdf,tabularx,bbm,color,bm,amsfonts,epic,booktabs,multirow}
\usepackage[colorlinks]{hyperref}

\usepackage{caption}
\allowdisplaybreaks[4]
\usepackage[toc,page]{appendix}

\usepackage{geometry}
\makeatletter

\newcommand{\Rmnum}[1]{\expandafter\@slowromancap\romannumeral#1@}
\makeatother

\numberwithin{equation}{section}
\newtheorem{theorem}{Theorem}[section]
\newtheorem{lemma}{Lemma}[section]

\newtheorem{remark}{Remark}[section]

\graphicspath{{Figs/}}

\begin{document}

\title{\large\bf On the structure of the $h$-fold sumsets}
\author{\large Jun-Yu Zhou, Quan-Hui Yang\footnote{Emails:~yangquanhui01@163.com. This
author was supported by the National Natural Science Foundation
for Youth of China, Grant No. 11501299, the Natural Science
Foundation of Jiangsu Province, Grant Nos. BK20150889,~15KJB110014.}}
\date{} \maketitle
 \vskip -3cm
\begin{center}
\vskip -1cm { \small
1. School of Mathematics and Statistics, Nanjing University of Information \\
Science and Technology, Nanjing 210044, China }
 \end{center}

\begin{abstract} Let~$A$ be a set of nonnegative integers. Let~$(h A)^{(t)}$
be the set of all integers in the sumset~$hA$ that have at least~$t$ representations
as a sum of~$h$ elements of~$A$.
In this paper, we prove that, if~$k \geq 2$, and~$A=\left\{a_{0}, a_{1}, \ldots, a_{k}\right\}$
is a finite set of integers such that~$0=a_{0}<a_{1}<\cdots<a_{k}$
and $\gcd\left(a_{1}, a_2,\ldots, a_{k}\right)=1,$ then there exist integers
~$c_{t},d_{t}$ and sets~$C_{t}\subseteq[0, c_{t}-2]$, $D_{t} \subseteq[0, d_{t}-2]$ such that
$$(h A)^{(t)}=C_{t} \cup\left[c_{t}, h a_{k}-d_{t}\right] \cup\left(h a_{k-1}-D_{t}\right)
$$
for all~$h \geq\sum_{i=2}^{k}(ta_{i}-1)-1.$ This improves a recent result of Nathanson
with the bound  $h \geq (k-1)\left(t a_{k}-1\right) a_{k}+1$.


\end{abstract}

\section{Introduction}
Let~$A$~and~$B$~be sets of integers. The sumsets and difference sets are defined by
$$
A+B=\{a+b: a \in A, b \in B\},\quad A-B=\{a-b: a \in A, b \in B\}$$
respectively. For any integer~$t,$ we define the sets
$$t+A=\{t\}+A,\quad t-A=\{t\}-A.$$
For~$h \geq 2$,~we denote by~$h A$~the~$h$-fold sumset of~$A$,~which is the set of all integers~$n$~of the form~ $n=a_{1}+a_{2}+\cdots+a_{h}$,~where~$a_{1}, a_{2}, \ldots, a_{h}$~are elements of~$A$~and not necessarily distinct.

In \cite{Nathanson72,{Nathanson96}}, Nathanson proved the following fundamental beautiful
theorem on the structure of $h$-fold sumsets.

{\bf Nathanson's Theorem A.} Let $A=\{a_0,a_1,\ldots,a_k\}$ be a finite set of integers such that $$0=a_0<a_1<\cdots<a_k\quad \text{and}\quad \gcd(A)=1.$$
Let $h_1=(k-1)(a_k-1)a_k+1$. There are nonnegative integers $c_1$ and $d_1$ and
finite sets $C_1$ and $D_1$ with $C_1\subseteq [0,c_1-2]$ and $D_1\subseteq [0,d_1-2]$
such that $$hA=C_1\cup [c_1,ha_k-d_1]\cup (ha_k-D_1)$$ for all $h\ge h_1$.

Later, Wu, Chen and Chen \cite{Wu} improved the lower bound of $h_1$ to
$\sum_{i=2}^k a_i-k$. Recently, Granville and Shakan \cite{GS}, and Granville and Walker \cite{GW}
gave some further results on this topic.

Let $A$ be a set of integers. For every positive integer $h,$ the $h$-fold representation function $r_{A, h}(n)$~counts the number of representations of $n$ as the sum of $h$ elements of $A.$ Thus,
$$
r_{A, h}(n)=\sharp\left\{\left(a_{j_{1}}, \ldots, a_{j_{h}}\right) \in A^{h}: n=\sum_{i=1}^{h} a_{j_{i}} \text { and } a_{j_{1}} \leq \cdots \leq a_{j_{h}}\right\}.
$$
For every positive integer~$t,$~let~$(h A)^{(t)}$~be the set of all integers~$n$~that have at least $t$ representations as the sum of $h$ elements of~$A,$~that is,
$$
(h A)^{(t)}=\left\{n \in \mathbf{Z}: r_{A, h}(n) \geq t\right\}
$$

Recently, Nathanson \cite{Nathanson} found that the sumsets $(hA)^{(t)}$ have the same structure as the sumset
$hA$ and proved the following theorem.

{\bf Nathanson's Theorem B.} Let $k\ge 2$, and let $A=\{a_0,a_1,\ldots,a_k\}$ be a finite set of integers
such that $0=a_0<a_1<\cdots<a_k$ and $\gcd(A)=1$. For every positive integer
$t$, let $h_t=(k-1)(ta_k-1)a_k+1$. There are nonnegative integers $c_t$ are $d_t$
and finite sets $C_t$ and $D_t$ with $C_t\subseteq [0,c_t-2]$ and $D_t\subseteq [0,d_t-2]$
such that $$(hA)^{(t)}=C_t\cup [c_t,ha_k-d_t]\cup (ha_k-D_t)$$ for all $h\ge h_t$.

In this paper, motivated by the idea of Wu, Chen and Chen \cite{Wu}, we improved the lower bound of $h$ in Nathanson's Theorem B.

\begin{theorem}\label{thm1}
Let $k \geq 2,$ and let $A=\left\{a_{0}, a_{1}, \ldots, a_{k}\right\}$ be a finite set of integers such that
$$
0=a_{0}<a_{1}<\cdots<a_{k} \quad \text { and } \quad \operatorname{gcd}(A)=1
$$
For every positive integer $t,$ let
$$
h_{t}=\sum_{i=2}^{k} (ta_{i}-1)-1
$$
There are nonnegative integers $c_{t}$ and $d_{t}$ and finite sets $C_{t}$ and $D_{t}$ with
$$
C_{t} \subseteq\left[0, c_{t}-2\right] \quad \text { and } \quad D_{t} \subseteq\left[0, d_{t}-2\right]
$$
such that
\begin{eqnarray}\label{eq111}
(h A)^{(t)}=C_{t} \cup\left[c_{t}, h a_{k}-d_{t}\right] \cup\left(h a_{k}-D_{t}\right)
\end{eqnarray}
for all $h \geq h_{t}$.
\end{theorem}
\begin{remark}\label{rem1} Theorem \ref{thm1} is optimal.
\end{remark}

We shall prove Theorem \ref{thm1} and Remark \ref{rem1} in Section 3. In Section 2,
we give some lemmas.

\section{Some Lemmas}

\begin{lemma}\label{lem1}\cite[See Lemma 1]{Nathanson} Let $A$ be a set of integers.
For any positive integer $h$ and $t$, we have
$$(hA)^{(t)}+A\subseteq ((h+1)A)^{(t)}.$$
\end{lemma}

\begin{lemma}\label{lem2} Let $k\ge 2$,
and let $A=\{a_0,a_1,\ldots,a_k\}$ be a set of integers satisfying
$0=a_0<a_1<\cdots<a_k$ and $\gcd(A)=1.$ For every positive integer $t$, let
$h_t=\sum_{i=2}^k (ta_i-1)-1$ and $c'_t=\sum_{i=1}^{k-1}a_i(ta_{i+1}-1)$.
If $c'_t-a_k<n<c'_t$, then there exist at least $t$ distinct nonnegative
$k$-tuples $(x_{1,s},x_{2,s},\ldots,x_{k,s})~(1\le s\le t)$ satisfying
\begin{eqnarray*} n=x_{1,s}a_1+x_{2,s}a_2+\cdots+x_{k,s}a_k\end{eqnarray*}
and $x_{1,s}+x_{2,s}+\cdots+x_{k,s}\le h_t$ for $s=1,2,\ldots,t$.
\end{lemma}
\begin{proof}
Since $\gcd\left(a_{1}, \ldots, a_{k}\right)=1,$ there exist integers~$x_{1}, \ldots, x_{k}$ such that
$$n=x_{1} a_{1}+\cdots+x_{k} a_{k}.$$
For any positive integer $s$, $[(s-1)a_{2},sa_{2}-1]$ is a complete residue system modulo $a_{2}$. Hence
there exists an integer $q$ such that
$x_{1}=a_{2} q+x_{1,s}$ with $(s-1)a_{2}\leq x_{1, s} \leq sa_{2}-1$.
This gives
$$n=x_{1, s} a_{1}+\left(a_{1}q+x_{2}\right) a_{2}+ \cdots +x_{k} a_{k}.$$
Let $x'_{2}=a_{1}q+x_{2}$. Similarly, there exists an integer $q'$ such that
$x'_{2}=a_{3} q'+x_{2,s}$ with $(s-1)a_{3}\leq x_{2, s} \leq sa_{3}-1$.
Now we have $$n=x_{1,s}a_1+x_{2,s}x_2+(a_2q'+x_3)a_3+\cdots +x_{k} a_{k}.$$
By continuing this process, we obtain
$$n=x_{1,s}a_{1}+x_{2,s}a_{2}+\cdots+x_{k,s}a_{k}$$
with $(s-1)a_{i+1} \leq x_{i,s} \leq sa_{i+1}-1$ for $i=1, \cdots, k-1$ and $x_{k,s}$ is some integer.
Hence, for any integer $s\in [1,t]$, we have
$$0\leq x_{i,s} \leq ta_{i+1}-1.$$
Since $n>c'_{t}-a_{k}$, it follows that
\begin{eqnarray*}
x_{k, s} a_{k}&=&n-\left(x_{1, s} a_{1}+x_{2, s} a_{2}+\cdots+x_{k-1, s} a_{k-1}\right) \\
& \geq &n-\left(t a_{2}-1\right) a_{1}-\cdots-\left(t a_{k}-1\right) a_{k-1} =n-c'_{t}>-a_{k},
\end{eqnarray*}
and then $x_{k,s}>-1$. Noting that $x_{k,s}$ is an integer, we have $x_{k,s} \geq 0$.
By the bound of $x_{i,s}$, the following nonnegative $k$-tuples
$$(x_{1, s},x_{2, s},\ldots,x_{k-1, s},x_{k, s})\quad (1\le s\le t)$$
are distinct.

Next, we shall prove that $x_{1, s}+x_{2, s}+\cdots+x_{k, s}\le h_t$ for $s=1,2,\ldots,t$.

For any integer $s\in [1,t]$, let $x_{1, s}+x_{2, s}+\cdots+x_{k, s}=u_s$.
Since $n<c'_{t}$, it follows that
$$\begin{aligned}
n &=x_{1, s} a_{1}+x_{2, s} a_{2}+\cdots+x_{k, s} a_{k} \\
&=x_{1, s} a_{1}+\cdots+x_{k-1, s} a_{k-1}+\left(u_s-x_{1, s}-x_{2, s}-\cdots-x_{k-1, s}\right) a_{k} \\
&=u_s a_{k}-x_{1, s}\left(a_{k}-a_{1}\right)-\cdots-x_{k-1, s}\left(a_{k}-a_{k-1}\right) \\
& \geq u_s a_{k}-\left(t a_{2}-1\right)\left(a_{k}-a_{1}\right)-\cdots-\left(t a_{k}-1\right)\left(a_{k}-a_{k-1}\right) \\
&=u_s a_{k}-a_{k}\left[\left(t a_{2}-1\right)+\cdots+\left(t a_{k}-1\right)\right]+a_{1}\left(t a_{2}-1\right)+\cdots+a_{k-1}\left(t a_{k}-1\right) \\
&=u_s a_{k}-\left(h_{t}+1\right) a_{k}+c'_{t} \\
&>u_s a_{k}-\left(h_{t}+1\right) a_{k}+n.
\end{aligned}$$
Hence $u_s a_{k}-\left(h_{t}+1\right) a_{k}<0$, and then $u_s<h_{t}+1$. Therefore, $u_s\leq h_{t}$.

This completes the proof of Lemma \ref{lem2}.
\end{proof}

\begin{lemma}\label{lem3} Let $c'_{t}$ and $h_t$ be defined in Lemma \ref{lem2}. Then
$$c'_{t}=\sum_{i=1}^{k-1}a_{i}\left(t a_{i+1}-1\right) \in\left(\left(h_{t}+1\right) A\right)^{(t)}.$$
\end{lemma}
\begin{proof} For $i=1,2,\ldots,k-1$, let $p_{i}=ta_{i+1}-1$. Then
\begin{eqnarray*}
c'_{t}=\left(t a_{2}-1\right) a_{1}+\cdots+\left(t a_{k}-1\right) a_{k-1}=p_{1} a_{1}+\cdots+p_{k-1} a_{k-1}.
\end{eqnarray*}
Noting that
$$
p_{1}+\cdots+p_{k-1}=\sum_{i=2}^{k} (ta_{i}-1)=h_{t}+1,
$$
we have $c'_{t} \in\left(h_{t}+1\right) A$.

Moreover, for any integers $r \in [0,t-1]$, we have
\begin{eqnarray*}
c'_{t} &=&\sum_{i=1}^{k-1}\left(t a_{i+1}-1\right)a_{i}
=\sum_{i=1}^{k-1}\left((t-r) a_{i+1}-1\right)a_{i}+r\sum_{i=1}^{k-1}a_ia_{i+1}\\
&=&\left((t-r) a_{2}-1\right)a_{1}+\sum_{i=2}^{k-1}((t-r)a_{i+1}-1+ra_{i-1})a_i+r a_{k-1}a_{k}\\
&:=&p_{1,r}a_{1}+p_{2,r}a_{2}+\cdots+p_{k-1,r}a_{k-1}+p_{k,r}a_{k},
\end{eqnarray*}
where $p_{1,r}=(t-r)a_2-1$, $p_{k,r}=ra_{k-1}$ and $p_{i,r}=(t-r)a_{i+1}-1+ra_{i-1}~(2\le i\le k-1)$.
Hence $p_{i,r}\geq 0$ for all $i\in[1,k]$ and
\begin{eqnarray*}\sum_{i=1}^{k} p_{i,r}&=(t-r) a_{2}-1+(t-r) a_{3}-1+r a_{1}+\cdots+(t-r) a_{k}-1+r a_{k-2}+r a_{k-1} \\
&=h_{t}+1-r\left(a_{2}+\cdots+a_{k}\right)+r\left(a_{1}+\cdots+a_{k-1}\right) \\
&=h_{t}+1-r\left(a_{k}-a_{1}\right) \leq h_{t}+1.
\end{eqnarray*}
Thus, $r_{A,{h_t+1}}(c'_{t})\geq t$, and so $c'_{t}\in\left(\left(h_{t}+1\right) A\right)^{(t)}$.
\end{proof}

\begin{lemma}\label{lem4} Let $n$ and $a_{1}$, $a_{2}$ be positive integers with $\gcd(a_{1},a_{2})=1$.
For any positive integer $t$, if $n>ta_{1}a_{2}-a_{1}-a_{2}$, then the diophantine equation
\begin{eqnarray}\label{eq1}a_{1}x+a_{2}y=n\end{eqnarray}
has at least $t$ nonnegative integer solutions. The lower bound of $n$ is also best possible.
\end{lemma}

\begin{proof}
Suppose that $n>ta_{1}a_{2}-a_{1}-a_{2}$. Let $(x_{0},y_{0})$ be a solution of the equation (\ref{eq1}). Then
all the integer solutions of the equation (\ref{eq1}) is
\begin{eqnarray}\label{eqk}\begin{cases}
x=x_{0}+ka_{2}, \\
y=y_{0}-ka_{1},
\end{cases}
~k\in \mathbb{Z}.\end{eqnarray}
In order to $x\ge 0$ and $y\ge 0$, we only need $x>-1$ and $y>-1$, that is,
\begin{eqnarray}\label{ineq1}\frac{-1-x_{0}}{a_{2}}<k<\frac{y_{0}+1}{a_{1}}.\end{eqnarray}
Since
\begin{eqnarray*}
&&\frac{y_{0}+1}{a_{1}}-\frac{-1-x_{0}}{a_{2}}=\frac{a_{1}+a_{2}+a_{1}x_{0}+a_{2}y_{0}}{a_{1}a_{2}}\\
&=&\frac{a_{1}+a_{2}+n}{a_{1}a_{2}}>\frac{a_{1}+a_{2}+ta_{1}a_{2}-a_{1}-a_{2}}{a_{1}a_{2}}=t,
\end{eqnarray*}
there exist at least $t$ integers $k$ such that (\ref{ineq1}) holds.

Therefore, the equation (\ref{eq1}) has at least $t$ nonnegative integer solutions.

Now suppose that $l=ta_{1}a_{2}-a_{1}-a_{2}$. Then
~$(ta_2-1,-1)$ is a solution of (\ref{eq1}). Take $x_0=ta_2-1$ and $y_0=-1$.
Then (\ref{eqk}) becomes
$$\begin{cases}
x=ta_{2}-1-ka_{2}, \\
y=-1 +ka_{1},
\end{cases}
~k\in \mathbb{Z}.$$
Since $x \geq 0$ and $y \geq 0$, it follows that $1\le k \le t-1$.
Hence there exist at most $t-1$ nonnegative integer solutions.

This completes the proof of Lemma \ref{lem4}.
\end{proof}

\section{Proofs}

\begin{proof}[Proof of Theorem 1]
Let $c'_{t}=\sum_{i=1}^{k-1}a_i(ta_{i+1}-1)$. By Lemma \ref{lem2}, there exist the smallest integers
$c_{t}$ and $d_{t}$ satisfying
$$\left[c'_{t}-a_{k}+1, c'_{t}-1\right] \subseteq\left[c_{t}, h_{t} a_{k}-d_{t}\right]
\subseteq\left(h_{t} A\right)^{(t)}.$$
It follows that $c_{t}-1 \notin\left(h_{t} A\right)^{(t)}$ and $h_{t} a_{k}-d_{t}+1 \notin\left(h_{t} A\right)^{(t)}$. Additionally
\begin{eqnarray}\label{eq22}c_{t} \leq c'_{t}-a_{k}+1,
\end{eqnarray}
\begin{eqnarray}\label{eq23}c'_{t}-1 \leq h_{t} a_{k}-d_{t}.
\end{eqnarray}
Define the finite sets~$C_{t}$~and~$D_{t}$~by
$$
C_{t}=\left(h_{t} A\right)^{(t)} \cap[0, c_{t}-2]
$$
and
$$
h_{t} a_{k}-D_{t}=\left(h_{t} A\right)^{(t)} \cap\left[h_{t} a_{k}-(d_{t}-2), h_{t} a_{k}\right].
$$
Then
\begin{eqnarray}\label{eq24}\left(h_{t} A\right)^{(t)}=C_{t} \cup
\left[c_{t}, h_{t} a_{k}-d_{t}\right] \cup\left(h_{t} a_{k}-D_{t}\right).
\end{eqnarray}
Therefore, (\ref{eq111}) holds for $h=h_{t}$.

Now we prove (\ref{eq111}) by induction on $h$.
Suppose that (\ref{eq111}) holds for some $h \geq h_{t}$. Define
$$B^{(t)}=C_{t} \cup\left[c_{t},(h+1) a_{k}-d_{t}\right] \cup\left((h+1) a_{k}-D_{t}\right).$$

Firstly we prove that $B^{(t)} \subseteq((h+1) A)^{(t)}.$

Take an arbitrary integer $b \in B^{(t)}$.

Case 1. $b \in C_{t} \cup\left[c_{t}, h_{t} a_{k}-d_{t}\right]$. By (\ref{eq24}), we have
$$b \in\left(h_{t} A\right)^{(t)} \subseteq((h+1) A)^{(t)}.$$

Case 2. $b \in\left[c_{t}+a_{k},(h+1) a_{k}-d_{t}\right] \cup\left((h+1) a_{k}-D_{t}\right)$. It follows that
$$b-a_{k} \in\left[c_{t}, h a_{k}-d_{t}\right] \cup\left(h a_{k}-D_{t}\right) \subseteq(h A)^{(t)}.$$
Thus, $b \in(h A)^{(t)}+a_{k} \subseteq((h+1) A)^{(t)}$.

Case 3. $h_ta_k-d_{t}+1\le b\le c_{t}+a_k-1$. By (\ref{eq22}) and (\ref{eq23}), we have
\begin{eqnarray}\label{eq25}c_t+a_{k}-1 \leq c'_{t} \leq h_{t} a_{k}-d_t+1.\end{eqnarray}
Thus $b=c'_{t}$. By Lemmas \ref{lem1} and \ref{lem3}, we have
$$b=c'_{t} \in\left(\left(h_{t}+1\right) A\right)^{(t)} \subseteq ((h+1) A)^{(t)}.$$

Therefore, $B^{(t)} \subseteq((h+1) A)^{(t)}$.

Next we shall prove that $((h+1) A)^{(t)}\subseteq B^{(t)}$.
Take an arbitrary integer $a \in((h+1) A)^{(t)}$.

Case 1. $a=c'_t$. By (\ref{eq25}) and $h \geq h_{t}$, we have
$$
c_{t} \leq c'_{t} \leq h_{t} a_{k}-d_{t}+1 \leq(h+1) a_{k}-d_{t}.
$$
Hence $a=c'_{t} \in B^{(t)}$.

Case 2. $a\not=c'_t$ and $a \notin(h A)^{(t)}$. Since $a \in((h+1) A)^{(t)}$, there exist $t$ nonnegative integer $k$-tuples
$(x_{1, s}, x_{2, s}, \ldots x_{k, s})~(1\le s\le t)$ satisfying
$$
a=x_{1, s} a_{1}+x_{2, s} a_{2}+\cdots+x_{k, s} a_{k}\quad \text{and}\quad x_{1, s}+x_{2, s}+\cdots+x_{k, s}=h+1.
$$
Furthermore, we can get
\begin{eqnarray}\label{eq100} 0 \leq x_{i, s} \leq t a_{i+1}-1,~ 1\le i\le k-1,~1\le s\le t.\end{eqnarray}
Otherwise, without loss of generality, assume that $x_{1,1} \geq ta_{2}$, then for $j=1,2,\ldots,t$,
we have
$$
\begin{aligned}
a &=x_{1,1} a_{1}+x_{2,1} a_{2}+\cdots+x_{k,1} a_{k} \\
&=\left(x_{1,1}-ja_{2}\right) a_{1}+\left(x_{2,1}+ja_{1}\right) a_{2}+\cdots+x_{k,1} a_{k}.
\end{aligned}
$$
Noting that for $j=1,2,\ldots,t$,
$$(x_{1,1}-ja_{2})+(x_{2,1}+ja_{1})+x_{3,1}+\cdots+x_{k,1}=h+1-j\left(a_{2}-a_{1}\right)<h+1,
$$
we have $a\in (h A)^{(t)}$, a contradiction. Hence the inequality (\ref{eq100}) holds.

By (\ref{eq100}), for $s=1,2,\ldots,t$, we have
$$
\begin{aligned}
a &=x_{1, s} a_{1}+x_{2, s} a_{2}+\cdots+x_{k, s} a_{k} \\
&=x_{1, s} a_{1}+\cdots+x_{k-1, s} a_{k-1}+\left(h+1-x_{1, s}-x_{2, s}-\cdots-x_{k-1, s}\right) a_{k} \\
&=(h+1) a_{k}-x_{1, s}\left(a_{k}-a_{1}\right)-\cdots-x_{k-1, s}\left(a_{k}-a_{k-1}\right) \\
& \geq(h+1) a_{k}-\left(t a_{2}-1\right)\left(a_{k}-a_{1}\right)-\cdots-\left(t a_{k}-1\right)\left(a_{k}-a_{k-1}\right) \\
&=(h+1) a_{k}-a_{k}\left[\left(t a_{2}-1\right)+\cdots+\left(t a_{k}-1\right)\right]+a_{1}\left(t a_{2}-1\right)+\cdots+a_{k-1}\left(t a_{k}-1\right) \\
&=(h+1) a_{k}-\left(h_{t}+1\right) a_{k}+c'_{t} \\
& \geq c'_{t}.
\end{aligned}
$$
Since $a \neq c'_{t},$ it follows that $a \geq c'_{t}+1$. By (\ref{eq22}), we have
$a \geq c'_{t}+1 \geq c_{t}+a_{k}.$

If $x_{k, s}=0$ for some $s$ with $1\le s\le t$, by (\ref{eq100}), then
$$
h+1=x_{1, s}+x_{2, s}+\cdots+x_{k-1, s} \leq\left(t a_{2}-1\right)+\cdots+\left(t a_{k}-1\right)=h_{t}+1 \leq h+1.
$$
Hence $x_{i, s}=t a_{i+1}-1$ for $i=1,2,\ldots,k-1$, and so
$$
a=\left(t a_{2}-1\right) a_{1}+\cdots+\left(t a_{k}-1\right) a_{k-1}=c'_{t},
$$
a contradiction.

Hence $x_{k, s} \geq 1$ for all integers $s=1,2,\ldots,t$.

Therefore, $a-a_{k} \in(h A)^{(t)}$ and $a-a_{k} \geq c_{t}$. By the induction hypothesis,
$$
a \in a_{k}+\left[c_{t}, h a_{k}-d_{t}\right] \cup\left(h a_{k}-D_{t}\right)=\left[c_{t}+a_{k},(h+1) a_{k}-d_{t}\right] \cup\left((h+1) a_{k}-D_{t}\right) \subseteq B^{(t)}.
$$

Case 3. $a\not=c'_t$ and $a \in(h A)^{(t)}$. By the induction hypothesis,~we have
$$
(h A)^{(t)}=C_{t} \cup\left[c_{t}, h a_{k}-d_{t}\right] \cup\left(h a_{k}-D_{t}\right).
$$
Since $C_{t} \cup\left[c_{t},(h+1) a_{k}-d_{t}\right] \subseteq B^{(t)},$ we can suppose that $a>(h+1) a_{k}-d_{t}$. By $a \in(h A)^{(t)},$ there exist at least $t$ distinct nonnegative
$k$-tuples $(x_{1,s},x_{2,s},\ldots,x_{k,s})~(1\le s\le t)$ such that
$$
a=x_{1, s} a_{1}+x_{2, s} a_{2}+\cdots+x_{k, s} a_{k}
$$
and
$$
x_{1, s}+x_{2, s}+\cdots+x_{k, s} \leq h.
$$

By Lemma 2, assume that~$0 \leq x_{i, s} \leq t a_{i+1}-1$ for $i=1,2, \ldots, k-1$.
If $x_{k, s}\leq 0$, then by (3) we have
$$
\begin{aligned}
a &\leq x_{1, s} a_{1}+x_{2, s} a_{2}+\cdots+x_{k-1, s} a_{k-1} \\
& \leq a_{1}\left(t a_{2}-1\right)+\cdots+a_{k-1}\left(t a_{k}-1\right) \\
&=c'_{t} \leq h_{t} a_{k}-d_{t}+1 \\
& \leq\left(h_{t}+1\right) a_{k}-d_{t} \leq(h+1) a_{k}-d_{t},
\end{aligned}
$$
which contradicts with $a>(h+1) a_{k}-d_{t}.$ Therefore $x_{k,s} \geq 1$ and $a-a_{k} \in(h A)^{(t)}.$ Since $a>(h+1) a_{k}-d_{t},$ it follows that $a-a_{k} \in h a_{k}-D_{t}.$
Hence
$$
a \in(h+1) a_{k}-D_{t} \subseteq B^{(t)},
$$
and so $((h+1) A)^{(t)} \subseteq B^{(t)}$.

This completes the proof of Theorem 1.
\end{proof}

\begin{proof}[Proof of Remark \ref{rem1}]
Let $n\geq 3$ be an integer and $A=\{0,~n,~n+1\}$. By Theorem 1, there exist integers $c_{t}$, $d_{t}$ and sets $C_{t} \subseteq[0, c_{t}-2]$, $D_{t} \subseteq[0, d_{t}-2]$ such that
$$(h A)^{(t)}=C_{t} \cup\left[c_{t}, h a_{k}-d_{t}\right] \cup\left(h a_{k}-D_{t}\right)$$
for all $h\geq h_t=t(n+1)-2$.

For any integer $m\ge c_t$, choose an integer $h'\geq t(n+1)-2$
such that $h' a_{k}-d_{t}\ge m$, then we have $m\in (h'A)^{(t)}$.

Hence, there exist $t$ nonnegative integer tuples $(u_{i},v_{i})~(1\le i\le t)$ such that $m=u_{i}n+v_{i}(n+1)$.

On the other hand, there does not exist $t$ nonnegative integer tuples $(u_{i},v_{i})~(1\le i\le t)$ such that $c_{t}-1=u_{i}n+v_{i}(n+1)$. Otherwise, if exist, choose
$h>\max_{1\le i\le t}\{u_i+v_i\}$, then we have $c_t-1\in (hA)^{(t)}$, a contradiction.
Hence, by Lemma \ref{lem4},  it follows that
$c_{t}-1=ta_1a_2-a_1-a_2=tn(n+1)-n-(n+1)$, and then $c_t=tn(n+1)-2n$.

Let $p \in ((h_{t}-1)A)^{(t)}$. Then there exist $t$ nonnegative integer tuples $(u_{i},v_{i})~(1\le i\le t)$ such that $p=u_{i}n+v_{i}(n+1)$ and $u_1>u_2>\cdots>u_t$
are the maximal $t$ numbers in all the representations. Hence
$$
\begin{aligned}
p=u_{1}n+v_{1}(n+1)&=[u_{1}-(n+1)]n+(v_{1}+n)(n+1)\\
&=[u_{1}-2(n+1)]n+(v_{1}+2n)(n+1)\\
&=\cdots \\
&=[u_{1}-(t-1)(n+1)]n+[v_{1}+(t-1)n](n+1).
\end{aligned}
$$
It follows that $u_{t}=u_{1}-(t-1)(n+1)$, $v_t=v_{1}+(t-1)n$.
Noting that
$$u_{t}+v_{t} < u_{t-1}+v_{t-1} < \cdots < u_{1}+v_{1} \leq h_{t}-1,$$
we have
$$
\begin{aligned}
u_{t}+v_{t}&=u_{1}-(t-1)(n+1)+v_{1}+(t-1)n\\
&=u_{1}+v_{1}-(t-1)\leq h_{t}-1-(t-1)=tn-2.
\end{aligned}
$$
Hence, for every  $p \in ((h_{t}-1)A)^{(t)}$,
$$
\begin{aligned}
p&=u_{t}n+v_{t}(n+1)\leq (u_{t}+v_{t})(n+1)\leq (tn-2)(n+1)\\
&=tn(n+1)-2(n+1)< tn(n+1)-2n=c_{t}.
\end{aligned}
$$
By (\ref{eq111}), it follows that
$$
((h_{t}-1)A)^{(t)}\subseteq [0,tn(n+1)-2(n+1)].
$$
Therefore, (\ref{eq111}) cannot hold for $h=h_{t}-1$, and so Theorem \ref{thm1} is optimal.
\end{proof}

%


\end{document}